\documentclass[a4paper]{article}

\usepackage{amssymb,amsmath,amsthm}
\usepackage{mathrsfs}
\usepackage{enumerate} 
\usepackage[pdftex]{graphicx} 
\usepackage[hmargin=3cm,vmargin=3cm]{geometry} 
\usepackage[all,cmtip]{xy} 
\usepackage{multirow}
\usepackage{bbm}
\usepackage{float}
\usepackage{relsize}
\usepackage{longtable}
\usepackage{tikz}
\usetikzlibrary{matrix}
\usetikzlibrary{arrows,shapes}

\tikzstyle{white}=[circle,draw=black!100,fill=white!100,thick,inner sep=0pt,minimum size =2mm]
\tikzstyle{blank}=[circle,draw=white!100,fill=white!100,thick,inner sep=0pt,minimum size =2mm]

\usepackage{soul}

\newtheorem{Theorem}{Theorem}[section]
\theoremstyle{plain}
\newtheorem{Thm}[Theorem]{Theorem}

\newtheorem{Lem}[Theorem]{Lemma}

\newtheorem{Cor}[Theorem]{Corollary}

\theoremstyle{definition}
\newtheorem{Defn}[Theorem]{Definition}

\newtheorem{Exam}[Theorem]{Example}

\newtheorem{Rem}[Theorem]{Remark}

\def\NN{\mathbb N}

\def\RR{\mathbb R}
\def\SS{\mathbb S}

\def\11{\mathbbm{1}}

\newcommand{\seq}[1]{\left\{#1\right\}}
\newcommand{\bra}[1]{\left(#1\right)}
\newcommand{\ip}[2]{\left\langle #1, #2 \right\rangle}

\newcommand{\abs}[1]{\left|#1\right|}

\renewcommand{\epsilon}{\varepsilon}

\def\diam{\text{diam}}

\def\diam{\text{diam}}

\makeatletter
\renewcommand*\env@matrix[1][*\c@MaxMatrixCols c]{%
  \hskip -\arraycolsep
  \let\@ifnextchar\new@ifnextchar
  \array{#1}}
\makeatother

\begin{document}

\title{On the supremal $p$-negative type of a finite metric space}
\author{Stephen S\'{a}nchez\footnote{The University of New South Wales, Sydney, Australia \newline email: Stephen.Sanchez@unsw.edu.au}}
\date{}

\maketitle{}

\begin{abstract}
We study the supremal $p$-negative type of finite metric spaces. An explicit expression for the supremal $p$-negative type $\wp \bra{X,d}$ of a finite metric space $\bra{X,d}$ is given in terms its associated distance matrix, from which the supremal $p$-negative type of the  space may be calculated. The method is then used to give a straightforward calculation of the supremal $p$-negative type of the complete bipartite graphs $K_{n,m}$ endowed with the usual path metric. A gap in the spectrum of possible supremal $p$-negative type values of path metric graphs is also proven.
\end{abstract}


\section{Introduction and Definitions}

The notions of negative type and generalized roundness have been used extensively in metric geometry to investigate various questions of embeddability. An early example is Shoenberg's classical result that a metric space is isometric to a subset of a Euclidean space if and only if it has 2-negative type \cite{Sch37}. This was later generalized to $L_p$ spaces by Bretagnolle, Dacunha-Castelle and Krivine \cite{BDCK67}, who showed that for $0< p \leq 2$, a real normed space is linearly isometric to a subset of some $L_p$ space if and only if it has $p$-negative type. In a different direction, Enflo \cite{Enf69} introduced the ideas of roundness and generalized roundness to answer in the negative a question of Smirnov: ``Is every separable metric space uniformly homeomorphic to a subset of $L_2[0,1]$?'' More recently, using an argument based on Enflo's ideas, Dranishnikov, Gong, Lafforgue and Yu \cite{DGLY02} showed that the answer to Smirnov's question is still negative even if one allows coarse embeddings, answering a question of Gromov.

Given the usefullness of negative type and generalized roundness, it is natural that one would want to calculate the supremal $p$-negative type (or maximal generalized roundness) for a given metric space $(X,d)$. Until now, this has remained a hard problem, even for spaces of relatively few points. The main result of this paper (Corollary~\ref{value}) shows that it is possible to calculate the supremal $p$-negative type for a given finite metric space, and provides a method for doing so. We first recall some work on negative type and generalized roundness, and then detail the main results.

In 1997 Lennard, Tonge and Weston \cite{LTW97} showed that the notions of negative type and generalized roundness coincide: a metric space $\bra{X,d}$ has $p$-negative type if and only if it has generalized roundness $p$. This has allowed a wider variety of techniques to be used in investigations of negative type and generalized roundness. We now define these concepts.

\begin{Defn}
Let $(X,d)$ be a metric space and $p \geq 0$. Then:
\begin{enumerate}[(i)]
 \item $(X,d)$ has \emph{$p$-negative type} if and only if for all natural numbers $k \geq 2$, all finite subsets $\seq{x_1,\ldots,x_k} \subset X$, and all choices of real numbers $\alpha_1, \ldots, \alpha_k$ with $\alpha_1 + \cdots + \alpha_k = 0$, we have:
\begin{equation}\label{type}
\sum_{1 \leq i, j \leq k} d(x_p,x_j)^p \alpha_i \alpha_j \leq 0.
\end{equation}
 \item $(X,d)$ has \emph{strict $p$-negative type} if and only if it has $p$-negative type and the inequalities \eqref{type} are all strict except in the trivial case $\bra{\alpha_1,\ldots, \alpha_k} = \bra{0, \ldots, 0}$.
\end{enumerate}
\end{Defn}

It is well known (see \cite[p.~11]{WW75}) that negative type (and so also generalized roundness) possesses the following interval property: if a metric space $\bra{X,d}$ has $p$-negative type, then it has $q$-negative type for all $0 \leq q \leq p$. We may thus define the \emph{supremal $p$-negative type} $\wp \bra{X,d}$ of a metric space $\bra{X,d}$ by
\[
\wp \bra{X,d} = \sup \seq{p : (X,d) \text{ has } p\text{-negative type}}.
\]
Moreover if $\wp(X,d)$ is finite then it is easy to see that $(X,d)$ does actually have $\wp(X,d)$-negative type. We write $\wp(X)$ when the metric is clear, and simply $\wp$ if the metric space is clear from context. 

Rather than give the original definition of generalized roundness, we will present an equivalent reformulation that is due to Lennard, Tonge and Weston \cite{LTW97} and Weston \cite{Wes95}. We first require some terminology.

\begin{Defn}
Let $s,t$ be arbitrary natural numbers and $X$ be any set. A \emph{normalized $(s,t)$-simplex in X} is a vector $\bra{a_1,\ldots, a_s, b_1, \ldots, b_t} \in X^{s+t}$ of $s+t$ distinct points, along with a load vector $\omega = \bra{m_1,\ldots, m_s,n_1,\ldots,n_t} \in \RR^{s+t}_+$ that assigns a positive weight $m_j >0$ or $n_i>0$ to each point $a_j$ or $b_i$ respectively and satisfies the following two normalizations
\[
m_1 + \cdots + m_s = 1 = n_1 + \cdots + n_t.
\]
Such a normalized $(s,t)$-simplex will be denote by $[a_j\bra{m_j};b_i \bra{n_i}]_{s,t}$.
\end{Defn}

Using the above we can now define the generalized roundness of a metric space $(X,d)$.

\begin{Defn}\label{genround}
 Let $(X,d)$ be a metric space and $p \geq 0$. Then:
\begin{enumerate}[(i)]
 \item $(X,d)$ has \emph{generalized roundness $p$} if and only if for all $s,t \in \NN$ and all normalized $(s,t)$-simplices $\left[a_j \bra{m_j} ; b_i \bra{n_i} \right]_{s,t}$ in $X$, we have
\begin{equation}\label{roundness}
\sum_{1 \leq j_i < j_2 \leq s} m_{j_1} m_{j_2} d \bra{a_{j_1}, a_{j_2}}^p + \sum_{1 \leq i_1 < i_2 \leq t} n_{i_1} n_{i_2}d \bra{b_{i_1}, b_{i_2}}^p \leq \sum_{j,i = 1}^{s,t} m_j n_i d \bra{a_j, b_i}^p. 
\end{equation}
\item $(X,d)$ has \emph{strict generalized roundness $p$} if and only if it has generalized roundness $p$ and the associated inequalities \eqref{roundness} are all strict.
\end{enumerate}
\end{Defn}

For consistency, we will phrase everything in terms of $p$-negative type, though referring to normalized $(s,t)$-simplices occasionally. The formulation in Definition \ref{genround} has the advantage that  \eqref{roundness} can be interpreted geometrically: the sum of the $p$-distances from the $a_j$ to the $b_i$ is at least as big as the sums of the $p$-distances within the $a_j$ and the $b_i$, all with the appropriate weights.

For some metric spaces, it is not too difficult to calculate their $p$-negative type, such as $L_p[0,1]$ with $0 < p \leq 2$ (in which case we have $\wp\bra{L_p} =p$, see \cite[Corollary 2.6]{LTW97}). For less homogeneous metric spaces, such as finite metric trees, their supremal $p$-negative type may be more obscure.

The $p$-negative type qualities of finite metric spaces in particular have been investigated. In \cite{DM90}, Deza and Maehara determine the supremal $p$-negative type for a finite graph $G$ endowed with the `truncated metric' in terms of the characteristic polynomials of $G$ and $\bar{G}$ (though their results were not cast in terms of negative type). In \cite{Wes95}, Weston showed that the supremal $p$-negative type of a finite metric space $\bra{X,d}$ may be bounded below by a function depending only upon the number of points in the metric space. He later improved upon this in \cite{Wes09}, in which he showed that for a finite metric space $\bra{X,d}$ with $n$ points and scaled diameter $\mathfrak{D} = \bra{\diam\,X}/ \min \seq{d(x,y): x \neq y}$, we have
\[
\wp \bra{X} \geq  \frac{ \ln \bra{  1 / \bra{1- \Gamma} }}{ \ln \mathfrak{D} } \quad \text{where} \quad \Gamma = \frac{1}{2} \bra{\frac{1}{\lceil n/2 \rceil} + \frac{1}{\lfloor n/2 \rfloor}},
\]
and that the bound is sharp for spaces in which $\mathfrak{D} \in [1,2]$.

Other investigations have focused on more specific classes of metric spaces. In \cite{HLMT98}, Hjorth, Lison\u{e}k, Markvorsen and Thomassen showed that metric spaces of two, three or four points, as well as metric trees, and any finite subsets of the sphere $\SS^n$ that do not contain two pairs of antipodal points, all have strict $1$-negative type. Later, in \cite{HKM01}, Hjorth, Kokkendorff and Markvorsen showed that finite subspaces of hyperbolic spaces are of strictly negative type. Recently, Doust and Weston \cite{DW08a} gave a new proof that finite metric trees have strict 1-negative type, and from this showed that for any finite metric tree $T$ with $n \geq 3$ vertices we have
\[
\wp(T) \geq 1 + \frac{\ln \bra{1+ \frac{1}{(n-1)^3(n-2)}}}{\ln(n-1)}.
\]

Useful upper bounds on the supremal $p$-negative type of finite metric spaces tend to be less general. For example, for $n \geq 2$ the complete graph $K_n$ has strict $p$-negative type for all $p \in [0,\infty)$, showing that no upper bound can depend solely on the number of points in the space. Given a metric space $(X,d)$, one method of bounding $\wp(X)$ from above is to find a normalized $(s,t)$-simplex $[a_j \bra{m_j};b_i \bra{n_i}]_{s,t}$ in $X$ such that \eqref{roundness} fails to hold for some exponent $q>0$, from which it follows that $\wp(X) < q$. How to find normalized $(s,t)$-simplices such that \eqref{roundness} fails for relatively small exponents without resorting to a numerical search has, until now, been a difficult problem. For finite metric spaces with few points and some symmetries, one may sometimes employ optimisation procedures.

In spite of these general results, determining the supremal $p$-negative type $\wp \bra{X,d}$ of a given finite metric space $\bra{X,d}$ has remained a hard problem. Following from the work of \cite{LW10} and \cite{Wol10}, we show that the supremal $p$-negative type of a finite metric space may be recast in terms of matrix properties of distance matrix, allowing one to calculate (at least numerically) the supremal $p$-negative type of a given finite metric space. In particular, we show that (Corollary~\ref{value})
\[
\wp\bra{X,d} = \min \seq{ p \geq 0 : \det \bra{D_p} =0 \text{ or } \ip{D_p^{-1}\11}{\11}=0 },
\]
where $D_p$ is the distance matrix for $\bra{X,d}$ with each entry raised to power $p$.  Using this, we then give a simple proof that the supremal $p$-negative type of the complete bipartite graphs $K_{n,m}$ endowed with the usual path metric is given by (Theorem~\ref{bipartite})
\[
\wp \bra{K_{n,m}} = \log_2 \bra{ \frac{2nm}{2nm -n -m} }.
\]
By using Corollary~\ref{value} to calculate $\wp$ for some small graphs, we also show that (Theorem~\ref{gap}) for any path metric graph $G$, we have $\wp(G) \notin \bra{\log_2 \bra{2 + \sqrt{3}},2}$, and that these bounds are sharp.

\section{The supremal $p$-negative type of a finite metric space}

While the interval property of $p$-negative type has been known for some time, the behaviour of strict $p$-negative type has only recently been determined. In \cite{LW10}, Li and Weston completely classify all intervals on which a metric space may have strict $p$-negative type. Two of their results which we will use are the following.

\begin{Thm}\label{supremal}
The supremal $p$-negative type of a finite metric space cannot be strict.
\end{Thm}

\begin{Thm}\label{strictthm}
Let $(X,d)$ be a metric space. If $(X,d)$ has $p$-negative type for some $p>0$, then it must have strict $q$-negative type for all $q$ such that $0 \leq q <p$.
\end{Thm}

Thus, given a finite metric space $\bra{X,d}$, $\wp(X)$ is the smallest $p \geq 0$ such that $\bra{X,d}$ has $p$-negative type but does not have strict $p$-negative type.

We first reformulate the $p$-negative type condition for a finite metric space in terms of its distance matrix. Let $x_1, \ldots, x_n$ be a labelling of the elements of $X$. For $p \geq 0$, we define the \emph{$p$-distance matrix $D_p$} of $(X,d)$ by
\[
D_p = \big[ d(x_i,x_j)^p \big]_{i,j}.
\]
Let $\Pi_0$ denote the hyperplane $\seq{\alpha \in \RR^n : \ip{\alpha}{\11}=0}$, where $\11$ denotes the vector whose entries are all 1 and $\ip{\cdot}{\cdot}$ is the standard inner-product. Then $(X,d)$ has $p$-negative type if and only if 
\begin{equation}\label{neg}
\ip{D_p\alpha}{\alpha} \leq 0 \text{ for all } \alpha \in \Pi_0
\end{equation}
and has strict negative type if and only if
\begin{equation}\label{strict}
\ip{D_p\alpha}{\alpha} < 0 \text{ for all } \alpha \in \Pi_0 - \seq{0}.
\end{equation}
This follows since any sum of the form in \eqref{type} may be obtained by letting the appropriate entries of $\alpha$ be the $\alpha_i$ and the rest $0$. We refer to conditions \eqref{neg} and \eqref{strict} as $D_p$ being of negative type and strict negative type respectively. It is clear that given another labelling $x_1', \ldots, x_n'$ and associated distance matrix $D_p'$ the inequalities \eqref{neg} and \eqref{strict} hold for the same values of $p \geq 0$, and so the specific labelling is unimportant.

The following result of Wolf (see \cite[Theorem~3.1]{Wol10}, where it appears in a more general setting) gives a matrix characterization of strict $p$-negative type for a metric space of $p$-negative type.

\begin{Thm}\label{wolf}
Let $(X,d)$ be a finite metric space of $p$-negative type and $D_p$ an associated $p$-distance matrix. Then $(X,d)$ has strict $p$-negative type if and only if
\begin{enumerate}[(i)]
 \item $\det \bra{D_p} \neq 0$, and
 \item $\ip{D_p^{-1}\11}{\11}\neq 0$.
\end{enumerate}

\end{Thm}
\begin{proof}
Recall that $\Pi_0 = \seq{\alpha \in \RR^n : \ip{\alpha}{\11}=0}$. We similarly define $\Pi_1 = \seq{\alpha \in \RR^n : \ip{\alpha}{\11}=1}$. \\
First, suppose that $(X,d)$ has strict $p$-negative type (i.e.,  $D_p$ is of strict negative type). We show that if $D_p \alpha = 0$, then necessarily $\alpha =0$, and so $\det \bra{D_p} \neq 0$. If $\alpha \notin \Pi_0$, then $\frac{\alpha}{\ip{\alpha}{\mathbbm{1}}} \in \Pi_1$, and as $w = \bra{\frac{1}{2},\frac{1}{2},0 \ldots, 0} \in \Pi_1$, we have
\[
w - \frac{\alpha}{\ip{\alpha}{\mathbbm{1}}} \in \Pi_0.
\]
Since $D_p$ is of negative type, it follows that
\[
\ip{D_p \bra{w - \frac{\alpha}{\ip{\alpha}{\mathbbm{1}}}}}{w - \frac{\alpha}{\ip{\alpha}{\mathbbm{1}}}} \leq 0.
\]
Since $D_p\alpha=0$ and $D_p$ is symmetric, expanding out the above gives $\ip{D_p w}{w} \leq 0$. This is a contradiction since $\ip{D_pw}{w} = \frac{1}{2}d(x_1,x_2)^p > 0$. Thus, we must have $\alpha \in \Pi_0$. Since $D_p\alpha=0$, we clearly have $\ip{D_p\alpha}{\alpha}=0$. But $D_p$ is of strict negative type, so we necessarily have $\alpha=0$. Thus, the only solution to $D_p\alpha=0$ is $\alpha=0$, and so $\det \bra{D_p} \neq 0$. Next, let $\beta = D_p^{-1}\11$. Suppose $\beta \in \Pi_0$; then $\ip{D_p\beta}{\beta} = \ip{\11}{\beta}=0$. But as $D_p$ is of strict negative type this implies $\beta=0$, which is clearly not the case. Thus $\beta \notin \Pi_0$, and so $\ip{D_p^{-1}\11}{\11} \neq 0$.

Conversely, suppose that $\det \bra{D_p} \neq 0$ and $\ip{D_p^{-1}\11}{\11} \neq 0$. Let $\alpha\in \Pi_0$ be such that $\ip{D_p\alpha}{\alpha}=0$. We wish to show that $x=0$. Since $D_p$ is of negative type, the symmetric bilinear form
\[
\bra{x,y} = -\ip{D_px}{y}
\]
defines a semi-inner-product on $\Pi_0$. By the Cauchy-Schwartz inequality, for any $\beta \in \Pi_0$ we have
\[
\abs{\ip{D_p\alpha}{\beta}}^2 \leq \abs{\ip{D_p\alpha}{\alpha}}\cdot \abs{\ip{D_p\beta}{\beta}} = 0,
\]
which means we must have $D_p\alpha = \lambda \11$ for some $\lambda \in \RR$. Since $\alpha \in \Pi_0$, we have
\[
0 = \ip{\alpha}{\11} = \lambda \ip{D_p^{-1}\11}{\11}.
\]
Since $\ip{D_p^{-1}\11}{\11} \neq 0$ (by assumption), we must have $\lambda =0$ and so $\alpha=0$. Thus, $D_p$ is of strict negative type, and so $(X,d)$ has strict $p$-negative type.

\end{proof}

Combining the above result with Theorems \ref{supremal} and \ref{strictthm} gives a method for computing $\wp(X)$ for a given finite metric space $(X,d)$. Indeed, by Theorem \ref{strict}, $(X,d)$ has strict $p$-negative type for all $p \in [0,\wp)$. By Theorem \ref{wolf}, this means that $\det\bra{D_p} \neq 0$ and $\ip{D_p^{-1}\11}{\11} \neq 0$ for all such $p \in [0,\wp)$. But, as noted above, $(X,d)$ does not have strict $\wp$-negative type, and so either $\det \bra{D_{\wp}} =0$, or $D_{\wp}^{-1}$ exists but $\ip{D_{\wp}^{-1}\11}{\11}=0$. Thus, we have the immediate corollary.

\begin{Cor}\label{value}
Let $(X,d)$ be a finite metric space. Then
\[
\wp \bra{X,d} = \min \seq{p \geq 0 : \det \bra{D_p} =0 \text{ or } \ip{D_p^{-1}\11}{\11}=0},
\]
where $D_p$ is an associated $p$-distance matrix for $(X,d)$.
\end{Cor}

The above allows for the direct calculation of $\wp$ for a given finite metric space. Indeed, $\det(D_p)$ is always a continuous function of $p$. The function $\ip{D_p^{-1}\11}{\11}$ is always at least piecewise continuous: whenever $\det(D_p) \neq 0$ it is continuous, and either has a removable discontinuity or blows up when $\det(D_p) =0$. Thus, the first zeroes of $\det(D_p)$ and $\ip{D_p^{-1}\11}{\11}$ can be found at least numerically. While Corollary~\ref{value} only deals with finite metric spaces, it may be used to numerically investigate $p$-negative type qualities of certain infinite metric spaces, see for example \cite{CDW11}. Unfortunately, since the process involves variable determinants and matrix inverses, the running time of a program to find $\wp\bra{X,d}$ for a metric space $\bra{X,d}$ with many points quickly becomes a problem.

The proof of Theorem \ref{wolf} also gives us a vector $\alpha \in \Pi_0$ such that $\ip{D_\wp \alpha}{\alpha} =0$. Equivalently, it gives a normalized $(s,t)$-simplex $[a_j\bra{m_j};b_i \bra{n_i}]_{s,t}$ in $X$ such that
\[
\sum_{1 \leq j_i < j_2 \leq s} m_{j_1} m_{j_2} d \bra{a_{j_1}, a_{j_2}}^\wp + \sum_{1 \leq i_1 < i_2 \leq t} n_{i_1} n_{i_2}d \bra{b_{i_1}, b_{i_2}}^\wp = \sum_{j,i = 1}^{s,t} m_j n_i d \bra{a_j, b_i}^\wp.
\]
(See \cite[Theorem 2.4]{DW08a} for an explicit proof of the equivalence.) The existence of such simplices was first shown in \cite[Corollary 4.4]{LW10}, though the proof was non-constructive.

\begin{Cor}\label{simplex}
Let $(X,d)$ be a finite metric space with $\wp(X,d)< \infty$, and $D_p$ be an associated $p$-distance matrix. Then:
\begin{enumerate}[ (i) ]
 \item If $\det \bra{D_{\wp}}=0$, then for any $\alpha \in \ker \bra{D_{\wp}}$ we have $\alpha \in \Pi_0$ and $\ip{D_{\wp}\alpha}{\alpha}=0$.
 \item If $\det \bra{D_{\wp}} \neq 0$  but $\ip{D_{\wp}^{-1}\11}{\11}=0$, then $D_{\wp}^{-1}\11 \in \Pi_0$ and $\ip{ D_{\wp} \bra{D_{\wp}^{-1}\11} }{D_{\wp}^{-1}\11}=0$.
\end{enumerate}
\end{Cor}

\begin{proof}
Part (ii) is self-evident. For part (i), we only need note that, since $D_\wp$ is of negative type, the argument used in the proof of Theorem \ref{wolf} shows that an element in the kernel of $D_{\alpha}$ must be an element of $\Pi_0$.
\end{proof}

Next, we consider some simple examples which demonstrate the different possibilities in Corollary~\ref{value}. That is, some finite metric spaces have $\det \bra{D_{\wp}}=0$ while others have $\ip{D_{\wp}^{-1}\11}{\11}=0$. While it is clearly not possible for both to be zero at $p = \wp$, the quantity $\ip{D_p^{-1}\11}{\11}$ may approach 0 along with $\det\bra{D_p}$ as $p \to \wp$ (as is the case for finite subsets of $\RR$).

\begin{Exam}\label{examgraph}
 Consider the following graph $G_1$ endowed with the path metric.
\begin{figure}[H]
  \centering
  \begin{tikzpicture}
\node[white,label=below:$v_1$] (1) at (0,0) {};
\node[white,label=right:$v_3$] (2) at (45:2cm) {};
\node[white,label=left:$v_4$] (3) at (135:2cm) {};
\node[white,label=left:$v_5$] (4) at (-135:2cm) {};
\node[white,label=right:$v_2$] (5) at (-45:2cm) {};
\draw (1) -- (2);
\draw (1) -- (3);
\draw (1) -- (4);
\draw (1) -- (5);
\draw (3) -- (4);
\end{tikzpicture}
  \caption{$G_1$}
\end{figure}
\noindent Then $\det(D_p) \neq 0$ for all $p \geq 0$ but $\ip{D_p^{-1}\11}{\11}=0$ at $p=\log_2 \bra{\frac{13+\sqrt{105}}{8}} = 1.538\ldots$. So, by Corollary~\ref{value}, we have $\wp\bra{G_1} = \log_2 \bra{\frac{13+\sqrt{105}}{8}}$. Applying Corollary~\ref{simplex} to find a vector $\alpha \in \Pi_0$ such that $\ip{D_{\wp}\alpha}{\alpha}=0$ gives
\[
\alpha = \bra{-1, \frac{\sqrt{105}-9}{4},\frac{\sqrt{105}-9}{4}, \frac{11-\sqrt{105}}{4},\frac{11-\sqrt{105}}{4}}.
\]
\end{Exam}

\begin{Exam}\label{examtree}
Consider the following simple graph $G_2$ endowed with the path metric.

\begin{figure}[H]
  \centering
  \begin{tikzpicture}
\node[white,label=-70:$v_1$] (1) at (0,0) {};
\node[white,label=left:$v_2$] (2) at (120:2cm) {};
\node[white,label=left:$v_3$] (3) at (-120:2cm) {};
\node[white,label=below:$v_4$] (4) at (2,0) {};
\node[white,label=below:$v_5$] (5) at (4,0) {};
\draw (1) -- (2);
\draw (1) -- (3);
\draw (1) -- (4);
\draw (4) -- (5);
\end{tikzpicture}
  \caption{$G_2$}
\end{figure}
\noindent Then, the smallest positive roots of $\det \bra{D_p}$ and $\ip{D_p^{-1}\11}{\11}$ are $p = 1.826\ldots$ and $p=1.576\ldots$ respectively, and so $\wp\bra{G_2} = 1.576\ldots$. In this case we have that $\wp$ is the smallest positive solution to
\[
4 \cdot 12^p - 4 \cdot 9^p - 7 \cdot 8^p + 8 \cdot 4^p + 8 \cdot 3^p -4 =0.
\]
We are doubtful that a more explicit expression for $\wp$ may be found. Using an approximate value for $\wp$ and applying Corollary~\ref{simplex} gives the (approximate) weight vector
\[
\alpha = \bra{-1.000,0.351,0.351,0.204,0.094}.
\]
\end{Exam}

\begin{Exam}\label{cycle}
Consider the 5-cycle $G_3$ endowed with the path metric.

\begin{figure}[H]
  \centering
  \begin{tikzpicture}
\node[white,label=right:$v_5$] (1) at (18:2cm) {};
\node[white,label=above:$v_1$] (2) at (90:2cm) {};
\node[white,label=left:$v_2$] (3) at (162:2cm) {};
\node[white,label=below:$v_3$] (4) at (234:2cm) {};
\node[white,label=below:$v_4$] (5) at (306:2cm) {};
\draw (1) -- (2);
\draw (1) -- (5);
\draw (2) -- (3);
\draw (3) -- (4);
\draw (4) -- (5);
\end{tikzpicture}
  \caption{$G_3$}
\end{figure}
\noindent Then, $\det\bra{D_p}=0$ at $p=\log_2 \bra{\frac{3+\sqrt{5}}{2}} = 1.388\ldots$ while $\ip{D_p^{-1}\11}{\11} \neq 0$ for all $p \geq 0$. Thus, $\wp \bra{G_3}= \log_2\bra{\frac{3+\sqrt{5}}{2}}$. In this case we have $\dim \ker (D_{\wp})=2$, though this multitude of vectors comes from the symmetries of $G_3$. One such vector is
\[
\alpha = \bra{0, \frac{3-\sqrt{5}}{2},-\frac{\sqrt{5}-1}{2},\frac{\sqrt{5}-1}{2},-\frac{3-\sqrt{5}}{2}}.
\]
\end{Exam}

The above examples also show that even a realtively `simple' metric space may have quite a convoluted supremal $p$-negative type

In \cite{DW08a}, Doust and Weston provide an algorithm for finding the essentially unique normalized $(s,t)$-simplex for a finite metric tree that maximizes the difference between the left- and right-hand sides of \eqref{roundness} at $p=1$. We note briefly note that, in general, these simplices are not the ones given by Corollary~\ref{simplex}. Indeed, the weightings given by the algorithm in \cite{DW08a} are always rational for trees endowed with the path metric, while the weights given in Example~\ref{examtree} are clearly not.

\begin{Rem}
When numerically calculating the supremal $p$-negative type of a particular finite metric space $\bra{X,d}$, it is important to consider $\det \bra{D_p}$ as well as $\ip{D_p^{-1}\11}{\11}$. While the latter is clearly not defined if $\det \bra{D_p} =0$, this may manifest itself as a removable discontinuity not detected when numerically finding its zeroes. This occurs in Example~\ref{cycle}, where $\det(D_p)=0$ at $p=\wp=\log_2 \bra{\frac{3+\sqrt{5}}{2}}$, but
\[
\lim_{p \to \wp} \ip{D_p^{-1}\11}{\11} = \frac{5}{5+\sqrt{5}}.
\]
\end{Rem}

\section{Some applications to path metric graphs}

In this section we give two applications of Corollary~\ref{value} to the supremal $p$-negative type of path metric graphs. We calculate the supremal $p$-negative type for the complete bipartite graphs, and then finish by proving a gap in the possible supremal $p$-negative type gaps of path metric graphs.

In complete bipartite graphs the only distances involved are 0, 1 and 2, and so the variable $p$ only appears in $\det \bra{D_p}$ and $\ip{D_p^{-1}\11}{\11}$ in the form $2^p$. So, in contrast to Example~\ref{examtree}, we are able to express $\wp$ explicitly.

\begin{Theorem}\label{bipartite}
Let $n,m$ be positive natural numbers not both equal to 1. Let $K_{n,m}$ be the complete bipartite graph on $n+m$ vertices $v_1, \ldots, v_{n+m}$ with vertex bipartition $V_{1}= \seq{v_1, \ldots, v_n}$,$V_2 = \seq{v_{n+1}, \ldots, v_{n+m}}$, endowed with the path metric. Then
\[
\wp \bra{K_{n,m}} = \log_2 \bra{ \frac{2nm}{2nm - n - m} }.
\]
\end{Theorem}

The above result may be deduced from \cite[Example 2]{DM90}, though the proof is quite complicated. The case $m=1$, $n\geq 2$ was proven by Doust and Weston in \cite[Theorem 5.6]{DW08a} by finding a $(1,n)$-simplex for which \eqref{roundness} has equality for the smallest possible exponent. Theorem~\ref{bipartite} was also proven by Weston in \cite{Wes09} in the cases $\abs{n-m}=0,1$ using the method of Lagrange multipliers. Using Corollary~\ref{value} we provide a straightforward and self-contained proof below.

\begin{proof}
Using the given labelling, the associated $p$-distance matrix is given by
\[
\begin{array}{r@{\,}l}

    & 
      \begin{matrix}
       \mspace{15mu}\overbrace{\rule{2.65cm}{0pt}}^{n} &  \overbrace{\rule{2.65cm}{0pt}}^{m}
      \end{matrix}
      \\

D_p = & \begin{pmatrix}
         0      &  2^p   &  \cdots  &  2^p   &  1     &  \cdots & \cdots   &  1     \\
         2^p    & \ddots & \ddots   & \vdots & \vdots &         &          & \vdots \\
         \vdots & \ddots & \ddots   &  2^p   & \vdots &         &          & \vdots \\
         2^p    & \cdots & 2^p      & 0      & 1      & \cdots  &  \cdots  & 1      \\
         1      & \cdots & \cdots   & 1      & 0      & 2^p     & \cdots   & 2^p    \\
         \vdots &        &          & \vdots & 2^p    & \ddots  & \ddots   & \vdots  \\
         \vdots &        &          & \vdots & \vdots & \ddots  & \ddots   & 2^p     \\
         1      & \cdots & \cdots   & 1      & 2^p    & \cdots  & 2^p      & 0
        \end{pmatrix}.

\end{array}
\]
Let $\wp$ be the supremal $p$-negative type of $K_{n,m}$. Suppose that $\det \bra{D_{\wp}} =0$. Then, there is some non-zero $\alpha \in \Pi_0$ such that $D_{\wp} \alpha =0$. By writing out $D_{\wp}\alpha$ we see that $\alpha$ must satisfy
\begin{equation}\label{n}
2^\wp \sum_{i=1}^n \alpha_i + \sum_{i=n+1}^{n+m} \alpha_i = 2^\wp \alpha_1 = \cdots = 2^\wp \alpha_{n}
\end{equation}
and
\begin{equation}\label{m}
\sum_{i=1}^n \alpha_i + 2^\wp \sum_{i=n+1}^{n+m} \alpha_i = 2^\wp \alpha_{n+1} = \cdots = 2^\wp \alpha_{n+m}.
\end{equation}
So $\alpha_1 = \cdots = \alpha_n$ and $\alpha_{n+1} = \cdots = \alpha_{n+m}$. It is clear that $\alpha_1 \neq0$, as otherwise we would have $\alpha=0$. Rescaling if necessary we may assume $\alpha_1 = \frac{1}{n}$. Using the fact that $\alpha \in \Pi_0$, we see that we then have $\alpha_{n+1}= - \frac{1}{m}$. Then equations \eqref{n} and \eqref{m} imply that
\[
2^\wp -1 = \frac{2^\wp}{n} \quad \Rightarrow  \quad \wp = \log_2 \bra{\frac{n}{n-1}}
\]
and
\[
1-2^\wp = - \frac{2^\wp}{m} \quad  \Rightarrow \quad \wp = \log_2 \bra{\frac{m}{m-1} }.
\]
If $n \neq m$ then this is clearly a contradiction, in which case $\wp$ must be given instead by the first zero of $\ip{D_p^{-1}\11}{\11}$. We shall see that in the case $n=m$ the above does actually give $\wp$.

One can see that the inverse of $D_p$, when it exists, is given by

\[
\begin{array}{r@{\,}l}

    & 
      \begin{matrix}
       \mspace{15mu}\overbrace{\rule{2.35cm}{0pt}}^{n} &  \overbrace{\rule{2.35cm}{0pt}}^{m}
      \end{matrix}
      \\

D_p^{-1} = & \begin{pmatrix}
         a      &  b     &  \cdots  &  b     &  c     &  \cdots & \cdots   &  c      \\
         b      & \ddots & \ddots   & \vdots & \vdots &         &          & \vdots  \\
         \vdots & \ddots & \ddots   &  b     & \vdots &         &          & \vdots  \\
         b      & \cdots & b        & a      & c      & \cdots  &  \cdots  & c       \\
         c      & \cdots & \cdots   & c      & d      & e       & \cdots   & e       \\
         \vdots &        &          & \vdots & e      & \ddots  & \ddots   & \vdots  \\
         \vdots &        &          & \vdots & \vdots & \ddots  & \ddots   & e       \\
         c      & \cdots & \cdots   & c      & e      & \cdots  & e        & d
        \end{pmatrix},

\end{array}
\]
where
\[
\begin{array}{ccc}
a & = & {\displaystyle \frac{\bra{1-m}\bra{n-2}\bra{2^p}^2+m\bra{n-1}}{2^p \bra{\bra{n-1}\bra{m-1} \bra{2^p}^2 -nm } }},\\
b & = & {\displaystyle \frac{\bra{m-1}\bra{2^p}^2 - m}{2^p \bra{\bra{n-1}\bra{m-1} \bra{2^p}^2 -nm } }}, \\
c & = & {\displaystyle \frac{-2^p}{2^p \bra{\bra{n-1}\bra{m-1} \bra{2^p}^2 -nm } }   },\\
d & = & {\displaystyle \frac{\bra{1-n}\bra{m-2}\bra{2^p}^2+n\bra{m-1}}{2^p \bra{\bra{n-1}\bra{m-1} \bra{2^p}^2 -nm } }  },\\
e & = & {\displaystyle \frac{\bra{n-1}\bra{2^p}^2 - n}{2^p \bra{\bra{n-1}\bra{m-1} \bra{2^p}^2 -nm } }}.
\end{array}
\]
Thus we see that
\begin{align*}
\ip{D_{p}^{-1}\11}{\11} & = na + n(n-1)b + 2nmc + md + m(m-1)e\\
& = \frac{ \bra{2nm-n-m}2^p-2nm }{(n-1)(m-1)\bra{2^p}^2-nm}.
\end{align*}
So $\ip{D_{p}^{-1}\11}{\11} =0$ precisely when
\begin{equation}\label{nm}
p = \log_2 \bra{\frac{2nm}{2nm-n-m}}.
\end{equation}
So, for $n \neq m$ we have $\wp \bra{K_{n,m}} = \log_2 \bra{ \frac{2nm}{2nm-n-m}}$. If $n=m$, then \eqref{nm} simplifies to
\[
p = \log_2 \bra{\frac{n}{n-1} }.
\]
However, for this value of $p$ we have $(n-1)(m-1) \bra{2^p}^2 -nm =0$, and so $D_{p}^{-1}$ does not even exist. So, if $n=m$ we never have $\ip{D_p^{-1}\11}{\11}=0$, and so by the above we have $\wp = \log_2 \bra{\frac{n}{n-1} }$. Thus, for all $n,m$ positive integers both not equal to 1, we have
\[
\wp \bra{K_{n,m}} = \log_2 \bra{\frac{2nm}{2nm-n-m} }. \qedhere
\]
\end{proof}

It is straightforward to see that when Corollary~\ref{simplex} is applied to Theorem~\ref{bipartite}, the vector $\alpha$ such that $\ip{D_{\wp}\alpha}{\alpha}=0$ is given by
\[
\begin{array}{r@{\,}l}
    & 
      \begin{matrix}
       \mspace{15mu}\overbrace{\rule{1.4cm}{0pt}}^{n} & \!\!\! \overbrace{\rule{2.06cm}{0pt}}^{m}
      \end{matrix}
      \\

&\bra{ {\displaystyle \frac{1}{n}, \cdots, \frac{1}{n}, - \frac{1}{m}, \cdots, - \frac{1}{m} }}.
\end{array}
\]
The vector $\alpha$ is equivalent to the normalized $(n,m)$-simplex where the first $n$ vertices are $v_1, \ldots v_n$ each with weighting $\frac{1}{n}$, and the remaining $m$ vertices $v_{n+1},\ldots, v_{n+m}$ each with weighting $\frac{1}{m}$. So, for the complete bipartite graphs $K_{n,m}$, the extremal normalized simplex is exactly as one would imagine.

Finally, we show the existence of a nontrivial `gap' in the possible supremal $p$-negative type values for path metric graphs. An analogous result in terms of roundness was shown in \cite{HLMRR10}.

\begin{Thm}\label{gap}
If $G$ is a connected path metric graph, then $\wp(G) \notin \bra{\log_2 \bra{2 + \sqrt{3}},2}$.
\end{Thm}

The idea of the proof it to show that if $G$ is not complete, a path or a cycle, then we can find a metrically embedded subgraph in $G$ with supremal $p$-negative type at most $\log_2 \bra{2 + \sqrt{3}}$. We deal with one case first in a lemma.

\begin{Lem}\label{cyclelemma}
If $G$ is a cyclic graph on $2n+1$ vertices ($n \geq 2$) endowed with the path metric, then $\wp(G) \leq \log_2 3$.
\end{Lem}

We note that while the above bound is in no way optimal, it is strong enough for our purposes. (Indeed, it follows from \cite[Corollary 3.2]{HLMRR10} that $\wp$ tends to $1$ as $n \to \infty$.)

\begin{proof}
 Label the vertices of $G$ by $v_1, \ldots, v_{2n+1}$ in cyclic order. Consider the normalized $(2,2)$-simplex $[a_j \bra{m_j};b_i \bra{n_i}]_{2,2}$ in $G$ given by
\[
m_1 = m_2 = n_1 = n_2 = \frac{1}{2} \quad \text{and} \quad a_1 = v_1, \quad b_1 = v_2, \quad a_2 = v_n, \quad b_2 = v_{n+1}.
\]
For this simplex, inequality \eqref{roundness} becomes
\[
\frac{1}{4}n^2 + \frac{1}{4}n^2 \leq \frac{1}{4}1^p + \frac{1}{4}n^p + \frac{1}{4}1^p + \frac{1}{4} \bra{n-1}^p,
\]
or more simply,
\begin{equation}\label{cyclebound}
 \bra{n-1}^p - n^p + 2 \geq 0.
\end{equation}
We show that \eqref{cyclebound} fails for $p > \log_2 3$ for all $n \geq 2$.

Set $f_n(p) = (n-1)^p - n^p +2$. Then $f_n(0) = 2$ for all $n$. Also note that since $g(x)=x^p$ increases at a decreasing rate for $0< p <1$, we have $f_n(p) > 1$ for all $0 < p <1$. We also note that since $g_p(x) = x^p$, increases at an increasing rate for $p > 1$, we have
\[
m^p - (m-1)^p > n^p - (n-1)^p
\]
for all $m>n$, whenever $p > 1$. This implies that $f_m(p) < f_n(p)$ for all $m > n$ and $p > 1$. Since $f_2(p)<0$ for all $p > \log_2 3$, we see that for all $n \geq 2$, $f_n(p)$ is negative for some $p \leq \log_2 3$. Thus we have $\wp(G) \leq \log_2 3$. 
\end{proof}

We are now able to prove Theorem~\ref{gap}.

\begin{proof}[Proof of Theorem~\ref{gap}]
As the supremal $p$-negative type of an infinite metric space is the infimum of the supremal $p$-negative types of its metrically embedded finite subspaces, it is enough to consider only finite graphs. Let $G$ be a connected graph on $n$ vertices. By considering all possible cases, we show that $\wp(G) \notin \bra{\log_2 \bra{2 + \sqrt{3}},2}$.

If $n=1$, then \eqref{roundness} is vacuously satisfied, and so in this case we have $\wp(G)=\infty$.

For $n=2$ or $3$, the only possibilities are displayed below. Their supremal $p$-negative types are $\infty$, $2$ and $\infty$ respectively.
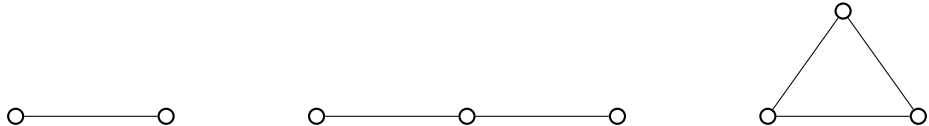
\begin{figure}[H]
  \centering
  \begin{tikzpicture}
\node[white] (1) at (0,0) {};
\node[white] (2) at (2,0) {};
\node[white] (3) at (4,0) {};
\node[white] (4) at (6,0) {};
\node[white] (5) at (8,0) {};
\node[white] (6) at (10,0) {};
\node[white] (7) at (12,0) {};
\node[white] (8) at (11,1.4) {};
\draw (1) -- (2);
\draw (3) -- (4);
\draw (4) -- (5);
\draw (6) -- (7);
\draw (6) -- (8);
\draw (7) -- (8);
\end{tikzpicture}
  \caption{Connected graphs on 2 or 3 vertices}
\end{figure}
From now on we assume that $n \geq 4$. There are two cases to consider, depending the on the maximum degree of $G$.

If the maximum degree of $G$ is 2, then $G$ is either a path of length $n-1$, or a cycle on $n$ vertices. If $G$ is a path, then it is easy to see that $\wp(G) =2$. If $G$ is a cyclic graph with $n$ even, then $\wp(G)=1$, since $G$ is embeddable in a sphere and contains two pairs of antipodal points (see \cite{HLMT98}). If $G$ is a cyclic graph with $n$ odd, then by Lemma~\ref{cyclelemma} we have $\wp(G)\leq \log_2 3 < \log_2 \bra{2+\sqrt{3}}$.

Suppose the maximum degree of $G$ is at least 3. If $G$ is a complete graph, then $\wp(G)=\infty$. So, from now on, we assume that $G$ is not complete. Since $G$ is connected, it must contain the subgraph $H_1$, displayed below. There are several cases to consider, depending on whether $v_2v_3$, $v_3v_4$ or $v_2v_4$ are edges in $G$.
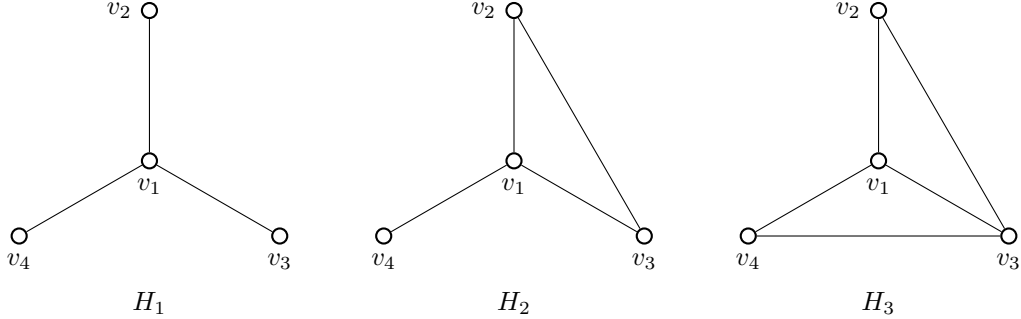
\begin{figure}[H]
  \centering
  \begin{tikzpicture}
\node[white,label=below:$v_1$] (1) at (0,0) {};
\node[white,label=left:$v_2$] (2) at (90:2cm) {};
\node[white,label=below:$v_3$] (3) at (-30:2cm) {};
\node[white,label=below:$v_4$] (4) at (-150:2cm) {};
\node[blank,label=below:$H_1$] (5) at (0,-1.5) {};
\draw (1) -- (2);
\draw (1) -- (3);
\draw (1) -- (4);
\end{tikzpicture}
\quad \quad
  \begin{tikzpicture}
\node[white,label=below:$v_1$] (1) at (0,0) {};
\node[white,label=left:$v_2$] (2) at (90:2cm) {};
\node[white,label=below:$v_3$] (3) at (-30:2cm) {};
\node[white,label=below:$v_4$] (4) at (-150:2cm) {};
\node[blank,label=below:$H_2$] (5) at (0,-1.5) {};
\draw (1) -- (2);
\draw (1) -- (3);
\draw (1) -- (4);
\draw (2) -- (3);
\end{tikzpicture}
\quad \quad
  \begin{tikzpicture}
\node[white,label=below:$v_1$] (1) at (0,0) {};
\node[white,label=left:$v_2$] (2) at (90:2cm) {};
\node[white,label=below:$v_3$] (3) at (-30:2cm) {};
\node[white,label=below:$v_4$] (4) at (-150:2cm) {};
\node[blank,label=below:$H_3$] (5) at (0,-1.5) {};
\draw (1) -- (2);
\draw (1) -- (3);
\draw (1) -- (4);
\draw (2) -- (3);
\draw (3) -- (4);
\end{tikzpicture}
\caption{Possible metrically embedded subgraphs in $G$}
\end{figure}
Suppose that none of $v_2v_3$, $v_3v_4$, $v_2v_4$ are edges in $G$. Then the graph $H_1$ is metrically embedded in $G$, and so $\wp(G) \leq \wp(H_1)$. By Theorem~\ref{bipartite}, we have $\wp\bra{H_1} = \log_2 3$. So, in this case, we have $\wp(G) \leq \log_2 3 < \log_2 \bra{2+\sqrt{3}}$.

Now suppose that exactly one of the possible edges $v_2v_3$, $v_3v_4$, $v_2v_4$ is present in $G$. Without loss of generality, let it be $v_2v_3$. Then, $H_2$ is metrically embedded in $G$. Using Corollary~\ref{value} it is straightforward to see that $\wp\bra{H_2} = \log_2 \bra{2+\sqrt{3}}$. So, in this case, we have $\wp(G) \leq \log_2 \bra{2 + \sqrt{3}}$.

Now suppose that exactly 2 of the possible edges $v_2v_3$, $v_3v_4$, $v_2v_4$ are present in $G$. Without loss of generality, let them be $v_2v_2$ and $v_3v_4$. Then $H_3$ is metrically embedded in $G$. Using Corollary~\ref{value} it is straightforward to see that $\wp \bra{H_3} = \log_2 3$. So, in this case, we have that $\wp(G) \leq \log_2 3 < \log_2 \bra{2+ \sqrt{3}}$.

If all three of the possible edges $v_2v_3$, $v_3v_4$, $v_2v_4$ are present in $G$, then, since $G$ is not a complete graph, there must be another vertex $v_5$ of $G$ adjacent to, say, $v_1$. It then follows to consider the possible edges $v_2v_5$, $v_3v_5$, $v_4v_5$.
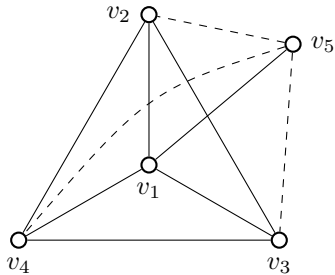
\begin{figure}[H]
  \centering
  \begin{tikzpicture}
\node[white,label=below:$v_1$] (1) at (0,0) {};
\node[white,label=left:$v_2$] (2) at (90:2cm) {};
\node[white,label=below:$v_3$] (3) at (-30:2cm) {};
\node[white,label=below:$v_4$] (4) at (-150:2cm) {};
\node[white,label=right:$v_5$] (5) at (40:2.5cm) {};
\draw (1) -- (2);
\draw (1) -- (3);
\draw (1) -- (4);
\draw (2) -- (3);
\draw (2) -- (4);
\draw (3) -- (4);
\draw (5) -- (1);
\draw[dashed] (5) -- (2);
\draw[dashed] (5) -- (3);
\draw[dashed] (5) .. controls (0,1) and (0,1) .. (4);

\end{tikzpicture}
  \caption{Possible edges to $v_5$}
\end{figure}
\noindent If any of the possible edges $v_2v_5$, $v_3v_5$, $v_4v_5$ are \emph{not} present, then it is clear that at least one of $H_2$ and $H_3$ is metrically embedded in $G$. So, by the above, we must have $\wp(G) \leq \log_2 \bra{2 + \sqrt{3}}$. If all of the edges $v_2v_5$, $v_3v_5$ and $v_4v_5$ are present, then, since $G$ is not complete, there must exist \emph{another} vertex $v_6$, adjacent to say $v_1$, whose possible edges we then consider. This process must eventually terminate as $G$ is finite and not complete. So, eventually, an embedded copy of $H_1$, $H_2$ or $H_3$ must be found in $G$. Thus, we must have $\wp(G) \leq \log_2 \bra{2 + \sqrt{3}}$.

So, in all possible cases, we either have $\wp(G) \leq \log_2 \bra{2 + \sqrt{3}}$ or $\wp(G) \geq 2$.

\end{proof}

\bibliographystyle{plain}
\bibliography{/c/z3206935/hdrive/Bibliography/bibliography.bib}

\end{document}